\documentclass{amsart}
\newcommand{\Q}{\mathbb{Q}}
\newcommand{\C}{\mathbb{C}}

\newcommand{\R}{\mathbb{R}}
\newcommand{\N}{\mathbb{N}}

\newcommand{\ran}{\operatorname{ran}}

\newcommand{\norm}[1]{\left\| #1 \right\|}

\newcommand{\unit}{\mathbf{1}}

\newcommand{\ratpoly}{\mathfrak{p}}
\newcommand{\ratpolyq}{\mathfrak{q}}

\theoremstyle{theorem}
\newtheorem{theorem}{Theorem}[section]
\newtheorem{lemma}[theorem]{Lemma}

\newtheorem{corollary}[theorem]{Corollary}

\newtheorem{proposition}[theorem]{Proposition}

\newtheorem{question}[theorem]{Question}

\theoremstyle{definition}

\numberwithin{equation}{section}

\begin{document}
\title{Computable Gelfand Duality}

% This is unfortunately what the style guide says to do when too many authors
\author[Burton et al.]{Peter Burton}
%\author[Burton]{Peter Burton}
\email{pburton@iastate.edu}

\author[]{Christopher J. Eagle}
%\author[Eagle]{Christopher J. Eagle}
\email{eaglec@uvic.ca}

\author[]{Alec Fox}
%\author[Fox]{Alec Fox}
\email{foxag@uci.edu}

\author[]{Isaac Goldbring}
%\author[Goldbring]{Isaac Goldbring}
\email{isaac@math.uci.edu}

\author[]{Matthew Harrison-Trainor}
%\author[Harrison-Trainor]{Matthew Harrison-Trainor}
\email{matthar@umich.edu}

\author[]{Timothy H. McNicholl}
%\author[McNicholl]{Timothy H. McNicholl}
\email{mcnichol@iastate.edu}

\author[]{Alexander Melnikov}
%\author[Melnikov]{Alexander Melnikov}
\email{alexander.g.melnikov@gmail.com}

\author[]{Teerawat Thewmorakot}
%\author[Thewmorakot]{Teerawat Thewmorakot}
\email{teerawat.thew@hotmail.com}

\thanks{The authors acknowledge the support of the 5-day workshop 23w5055 held at Banff International Research Station where this project was initiated. The second author was supported by NSERC Discovery Grant RGPIN-2021-02459.  The fifth author acknowledges support from the National Science Foundation under Grant No.\ \mbox{DMS-2153823}.}

%other authors

\begin{abstract}
We establish a computable version of Gelfand Duality. Under this computable 
duality, computably compact presentations of metrizable spaces uniformly effectively correspond to
	computable presentations of unital commutative $C^*$ algebras.
\end{abstract}
\maketitle

\section{Introduction}\label{sec:intro}

One of the most foundational results in the study of operator algebras is the Gelfand Duality Theorem.  By means of this theorem a commutative unital $C^*$ algebra $A$ can
be represented as the $C^*$ algebra of continuous functions from a compact metrizable space $X$ into the field of complex numbers $\C$.
Specifically, $X$ may be chosen to be the \emph{spectrum} $\Delta(A)$ of $A$ (the set of nonzero  
homomorphisms of $A$ into $\C$) with the Gelfand topology (the weakest topology in which all evaluation
maps are continuous).   By the Banach-Stone Theorem, the space $X$ is unique up to homeomorphism.
Here, we present a computable Gelfand duality within the framework of 
effective metric structure theory.  Specifically, we prove the following.  

\begin{theorem}\label{thm:main}
Suppose $X$ is a compact metrizable space.  If $C^*(X)$ is computably presentable, 
then $X$ has a computably compact presentation.
\end{theorem}

Moreover, our proof of Theorem \ref{thm:main} is nearly fully uniform.  Namely, it provides an effective 
operator that given a name of a (not necessarily computable) presentation of a $C^*$ algebra $A$ and a corresponding name of the unit
of $A$, produces a name of a presentation of $X$ and of a total boundedness function for $X$.  For the reader who is unfamiliar with computable presentations of metric structures (such as metric
spaces and Banach spaces), we give precise definitions for the setting of $C^*$ algebras in Section \ref{sec:back}.  These definitions are drawn from the recent work of A. Fox \cite{Fox.2022+}.  
For now, let us conceive of a computable presentation of a metric structure as a dense sequence
by means of which the operations and the metric can be computably approximated.  
A computably compact presentation of a metric space provides enough information to compute
arbitrarily tight covers of the space.  These presentations have a number of important features
not enjoyed by an arbitrary presentation.  For example, they provide for the computation of 
maxima of computable functions.  We discuss names of presentations
in Section \ref{sec:appl.ext}.  Roughly speaking, a name of a presentation is a point in Baire space that 
completely describes the presentation.

We note that the converse of Theorem \ref{thm:main} has been proven by A. Fox \cite{Fox.2022+}.  
\it We thus have a classification of the commutative unital $C^*$ algebras that are computably
presentable.\rm

Part of the significance of our work stems  from our proof of Theorem \ref{thm:main} which
gives a concise demonstration of 
the effectiveness of the Gelfand duality with only the rudiments of operator algebras and 
computable analysis.  The classical proofs of Gelfand duality typically rely on 
the Banach-Alaoglu Theorem.  The most effective version of this theorem that we are aware of is 
due to V. Brattka \cite{Brattka.2008}.  
 This version effectively embeds the closed unit ball of a dual into a computably compact 
presentation of a certain metric space.  However, 
the embedded image does not satisfy the computable compactness criterion considered here.  
Our proof of Theorem \ref{thm:main} is structured so as to avoid the need for 
an effective version of the Banach-Alaoglu Theorem.  In particular, we show that the information provided
by a presentation of $C^*(X)$ can be used to identify the points of $X$ 
by means of a certain family of \emph{vector names} which are judiciously constructed sequences
of vectors of $A$.    
In a sequence of papers, Banaschewski et. al. set forth a constructive proof of Gelfand duality in the 
context of toposes \cite{Banas.Mulvey.1997}, \cite{Banas.Mulvey.2000a}, \cite{Banas.Mulvey.2000b}, \cite{Banas.Mulvey.2006}, \cite{Coquand.Spitters.2009}.   By contrast, our proof of Theorem \ref{thm:main} takes place in 
a fairly concrete setting and requires only a minimal knowledge of computability and functional
analysis.   
   
Theorem \ref{thm:main} also advances the recently emerged program of effective metric structure theory which seeks to understand metric structures (such as Banach spaces) 
from the perspective of computable structure theory, that is, studying which structures have computable
presentations and identifying presentations that are essentially the same, that is, computably isomorphic.
We refer the reader
to the texts by Ash and Knight and Montalban for a much more expansive treatment of 
classical computable structure theory \cite{Ash.Knight.2000}, \cite{Montalban.2021}.  
The origins of effective metric structure theory
go back at least as far as the seminal work of Pour-El and Richards \cite{Pour-El.Richards.1989}.
However, the spark for the fairly recent development of the theory is the 
2013 paper of A. Melnikov \cite{Melnikov.2013}.  The insights in the latter led to a significant amount of work on the 
computable structure theory of Banach spaces, in particular Lebesgue spaces \cite{McNicholl.2017,Clanin.McNicholl.Stull.2019,Brown.McNicholl.2020} 
and Stone spaces \cite{bastone}.  Recently, 
A. Fox has extended this activity to the realm of $C^*$ algebras \cite{Fox.2022+}, and our efforts build on his.  

Perhaps most consequentially, Theorem \ref{thm:main} adds to the list of recently discovered 
\emph{computable dualities} such as the computable Stone duality and the computable Pontryagin duality
\cite{bastone}, \cite{Pontr}.  We discuss these and other computable dualities
in Section \ref{sec:back}.   Computable dualities have already led to the solution of several
open problems \cite{lupini,topsel,tdlc}.    The computability of the Gelfand duality connects 
the computability of $C^*$ algebras with the well-developed area of 
computably compact Polish spaces
\cite{EffedSurvey}.  
One would therefore expect Theorem \ref{thm:main} to lead to 
new discoveries in computable operator algebras.  Indeed, in Section \ref{sec:appl.ext},
 we combine Theorem \ref{thm:main} with known results in computable algebra and topology to produce 
interesting examples of $C^*$ algebras that do not have computable presentations.  

The paper is organized as follows.  Section \ref{sec:back} summarizes background from functional analysis and computable analysis. In Section \ref{sec:prelim.cl},
we attend to a few preliminary matters of a purely classical nature.  In particular,
we introduce the concept of a \emph{vector name} of a point and prove some classical 
properties of such names.  
In Section \ref{sec:prelim.comp}, we present preliminary results on the computability of the unit and related findings on 
the computability of certain post-composition operators.    These properties will then be used in the proof of 
Theorem \ref{thm:main} which is given in Section \ref{sec:proof.main}.
In Section \ref{sec:appl.ext}, we discuss some consequences of Theorem \ref{thm:main}.  We also
discuss its uniformity.  
Section \ref{sec:concl} summarizes our results and presents some directions for future work.

\section{Background}\label{sec:back}

We assume knowledge of the fundamentals of computability theory as expounded in \cite{Cooper.2004}.
Fix an effective enumeration $(\phi_e)_{e \in \N}$ of the computable partial functions from $\N$ into $\N$.

$\C$ is the field of scalars for each $C^*$ algebra considered herein.  $\N$ denotes the set of nonnegative integers.

We follow the computability theory of operator algebras developed by A. Fox  \cite{Fox.2022+}.   This 
framework is an extension of the computability theory for Banach spaces put forth by Pour-El and 
Richards \cite{Pour-El.Richards.1989}.  
  Fix a $C^*$ algebra $A$.    
We say $(A, (v_n)_{n \in \N})$ is a \emph{presentation} of $A$ if 
 $(v_n)_{n \in \N}$ generates a dense subalgebra of $A$.  
 If $A^\# = (A, (v_n)_{n \in \N})$ is a presentation of $A$, then each vector in 
 the subalgebra of $A$ generated by $(v_n)_{n \in \N}$ over the field of rational scalars is a 
 \emph{rational vector} of $A^\#$.
 
 By means of standard techniques, we 
 can generate an effective indexing of the rational vectors of a presentation.  Usually, 
 it is not necessary to provide the details of such an indexing, but for the sake of later developments 
 we will be more precise.   
 Specifically, we index the rational vectors of a presentation by means of rational $*$-polynomials as 
 follows.
 Let $x_0, x_1, \ldots$ be pairwise distinct indeterminants, and let $\mathcal{U}$ denote the free 
$*$-algebra generated by $X = \{x_0, x_1, \ldots \}$ over $\Q(i)$.  
Fix an effective enumeration $(\ratpoly_j)_{j \in \N}$ of $\mathcal{U}$.  By effective, we mean that from $m,n$ we can compute $j,k,s$ so that 
$\ratpoly_j = \ratpoly_m + \ratpoly_n$, $\ratpoly_k = \ratpoly_m\ratpoly_n$, and 
$\ratpoly_s = \ratpoly_m^*$.  
When $A^\# = (A, (u_n)_{n \in \N})$ is a presentation of $A$, and when $\ratpolyq \in \mathcal{U}$, let 
$\ratpolyq[A^\#]$ denote the vector of $A$ obtained from $\ratpoly$ by substituting $u_j$ for 
$x_j$ for each $j \in \N$.  We call $\ratpoly_j[A^\#]$ the \emph{$j$-th rational vector of $A^\#$}.  

 We note that given indices of rational vectors $u$ and $v$, it is possible
 to compute indices of $uv$, $u + v$, and $u^*$.   We also note that this indexing is not necessarily 
 injective, nor can we necessarily effectively determine if two numbers index the same rational vector.  
 
 A presentation $A^\#$ is \emph{computable} if the norm is computable on the rational vectors, that is, there is an algorithm that given $k \in \N$ and an index of a rational vector $v$ of $A^\#$, 
 computes a rational number $q$ so that $|q - \norm{v}| < 2^{-k}$. 
 $A$ is \emph{computably presentable} if it has a computable presentation.  
 
 The \emph{standard presentation} of the $C^*$ algebra $\C$ is defined by setting $v_n = 1$ for all $n \in \N$. The rational 
 vectors of this presentation are precisely the rational points of the plane.  We identify $\C$ with its standard presentation, and no other presentation of $\C$ is considered.
 
 Fix a presentation $A^\#$ of $A$, and let $v_0$ be a vector of $A$.  $v_0$ is a \emph{computable vector} of $A^\#$
 if there is an algorithm that given $k \in \N$ computes a rational vector $v$ of $A^\#$ so that 
 $\norm{v_0 - v} < 2^{-k}$.  
 An index of such an algorithm is an \emph{$A^\#$-index} of $v_0$.  
 A sequence $(u_n)_{n \in \N}$ of vectors of $A$ is a \emph{computable sequence} 
 of $A^\#$ if $u_n$ is a computable vector of $A^\#$ uniformly in $n$, that is, if there is an algorithm
 that given $n,k \in \N$ computes a rational vector $v$ of $A^\#$ so that $\norm{v - u_n} < 2^{-k}$.  
 
We define computability of operators and functionals via names as follows.  An \emph{$A^\#$-name} of $v_0$ is a 
sequence $(u_k)_{k \in \N}$ of rational vectors of $A$ so that 
 $\norm{u_k - v_0} < 2^{-k}$ for all $k \in \N$.  It follows that $v_0$ is a computable vector of $A^\#$ if and 
 only if $v_0$ has a computable name.  Since we identify $\C$ with its standard presentation, and as no other
 presentations of $\C$ are considered, 
we simply refer to a $\C$-name as a name.
 
 Suppose $T$ is an $n$-ary operator on $A$.  We say $T$ is a \emph{computable operator of $A^\#$}
 if there is an oracle Turing machine that given $A^\#$-names of vectors $v_1, \ldots, v_n$ computes an 
 $A^\#$-name of $T(v_1, \ldots, v_n)$.  An \emph{$A^\#$-index} of $T$ is an index of such a machine.  $T$ is \emph{intrinsically computable} if $T$ is a computable operator of every computable 
 presentation of $A$.  $T$ is \emph{uniformly intrinsically computable} if there is an algorithm 
 that given an index of a computable presentation $A^+$ produces an $A^+$-index of $T$.
 It is easily shown that if $T$ is a computable operator of $A^\#$, and if 
 $v_1, \ldots v_n$ are computable vectors of $A^\#$, then $T(v_1, \ldots, v_n)$ is a computable
 vector of $A^\#$.  Furthermore, an $A^\#$-index of $T(v_1, \ldots, v_n)$ can be computed from 
 $A^\#$ indices of $T$, $v_1$, $\ldots$, $v_n$. 
 
 Computability of functionals is defined similarly.  We also define intrinsically computable functionals and uniformly intrinsically
 computable functionals in the same way that we defined intrinsically computable operators and 
 uniformly intrinsically computable operators.   Again, it is easily shown that if $v_0$ is a computable
 vector of $A^\#$, and if $F$ is a computable functional of $A^\#$, then 
 $F(v_0)$ is a computable point of the plane.  Furthermore, from an $A^\#$-index of $v_0$ and 
 an $A^\#$-index of $F$, it is possible to compute an index of $F(v_0)$.
  
 It follows from these definitions that the addition, multiplication, and involution of $A$ are uniformly intrinsically computable operators
 of $A$.  In addition, for each $s \in \Q(i)$, the multiplication-by-$s$ operator is a uniformly intrinsically 
 computable operator of $A$ uniformly in $s$.  The norm is an intrinsically computable functional of $A$.  An additional useful principle is the following.  If 
 $A^\# = (A, (v_n)_{n \in \N})$ is a computable presentation, then a bounded 
 linear functional $F$ on $A$ is computable if $(F(v_n))_{n \in \N}$ is computable.
  
 We will use moduli of convergence to demonstrate the computability of certain limits.  These are 
 defined as follows.  If $(v_n)_{n \in \N}$ is a convergent sequence of vectors of $A$, then 
 a \emph{modulus of convergence for $(v_n)_{n \in \N}$} is a function $g : \N \rightarrow \N$
 so that $\norm{v_m - \lim_n v_n} < 2^{-k}$ whenever $m \geq g(k)$.  
 It is easily shown that if $(v_n)_{n \in \N}$ is a computable sequence of $A^\#$ that has a computable modulus of convergence, then $\lim_n v_n$ is a computable vector of 
 $A^\#$.  

 It is well-known that a computable function $g : \R^n \rightarrow \R$ can be 
 effectively approximated on compacta by rational polynomials.  This principle 
 does not hold for computable functions of one or more complex variables.
 However, as the involution on $\C$ gives access to the real and imaginary 
 parts of a complex number, we can nevertheless effectively approximate 
 a computable $g : \C^n \rightarrow \C$ on compacta by rational $*$-polynomials.
 
 Suppose $(X,d)$ is a complete metric space.  
 A \emph{computable presentation} of $(X,d)$ consists of a dense sequence $(p_n)_{n \in \N}$ of points 
 of $X$ so that the array $(d(p_m,p_n))_{m,n \in \N}$ is computable; that is, there is 
  an algorithm that given $m,n,k \in \N$ computes $q \in \Q$ so that $|q - d(p_m,p_n)| < 2^{-k}$.
  If $X$ is a Polish space, then a computable presentation of $X$ consists of specifying 
  a compatible complete metric $d$ and a computable presentation of $(X,d)$.  
 Suppose $X$ is a Polish space and $X^\# = (X, d, (p_n)_{n \in \N})$ is a computable presentation of $X$.  
 We say that $X^\#$ is \emph{computably compact} if from $k \in \N$ it is possible to compute
 $n_0, \ldots, n_t \in \N$ so that $X \subseteq \bigcup_j B(p_{n_j}; 2^{-k})$.  Since $(X,d)$ is complete, 
 if $X$ has a computably compact presentation, then $X$ is compact.
The terminology `computably compact' was coined by Mori, Tsujii, and Yasugi~\cite{MoriTsujiYasugi}.
As its name suggests, this notion is restricted to compact spaces, but it can be generalized to locally compact spaces; e.g.,~\cite{LocalPauly,lc2,separ}.  
One remarkable feature of this notion is that it is exceptionally robust; at least \emph{nine} equivalent 
formulations of computable compactness can be found in \cite{EffedSurvey,IlKi, Pauly}.

As mentioned in the introduction, Theorem \ref{thm:main} contributes to the program of computable
dualities.  These dualities include the following.
\begin{enumerate}
	\item[(D1)]  A Stone space $B$ has a computably compact presentation if and only if $C(B; \mathbb{R})$ is computably presentable \cite{bastone}.

	\item[(D2)]  A (discrete, countable) torsion-free abelian group $G$ is computably presentable if and only if its connected compact Pontryagin dual  $\widehat{G}$ has a computably compact presentation \cite{Pontr,lupini}.

	\item[(D3)]  If $T$ is a discrete, countable, and torsion Abelian group, then 
	the following are equivalent \cite{Pontr,EffedSurvey}. 
	\begin{enumerate}
		\item $T$ is computably presentable.
		
		\item  The profinite Pontryagin dual $\widehat{T}$ has a computably compact presentation.
		
		\item  $\widehat{T}$ has a recursively profinite presentation.
	\end{enumerate}

	\item[(D4)]  If $B$ is a countable discrete Boolean algebra, then the following
	are equivalent \cite{uptohom,topsel}.
	\begin{enumerate}
		\item  $B$ has a computable presentation.
		
		\item The Stone space $\widehat{B}$ of $B$ has a computably compact presentation.
		
		\item $\widehat{B}$ has a computable presentation.
	\end{enumerate}

	\item[(D5)]  The computably locally compact totally disconnected groups are exactly the duals of the computable (discrete, countable) meet groupoids of their compact cosets \cite{tdlc, separ}.

	\item[(D6)]  A profinite group is recursively presented if and only if 
	it is topologically isomorphic to the Galois group of a computable field extension \cite{MetNer79,LaRothesis,SmithThesis}.
\end{enumerate}

A few further dualities can be found in \cite{newpolish,EffedSurvey,lupini}.  We refer 
the reader to \cite{EffedSurvey} for a rather detailed exposition of some aspects of 
this unified theory. 

\section{Preliminaries from classical analysis}\label{sec:prelim.cl}

Fix a compact metrizable space $X$, and let $A = C^*(X)$.  
A \emph{vector name} of $p \in X$ is a sequence $(f_n)_{n \in \N}$ of vectors of $A$ so that 
$\{p\} = \bigcap_{n \in \N} f_n^{-1}(\frac{1}{2}, \infty)$.  It follows from Urysohn's Lemma that every point of $X$
has a vector name.  Our approach to proving Theorem \ref{thm:main} is to use vector names to identify points. 
However, as we shall see later, not every vector name lends itself to computability; we need to use names that 
have a certain amount of structure.  Accordingly, we define a vector name $(f_n)_{n \in \N}$ to be 
\emph{well structured} if $f_{n+1}^{-1}(\frac{1}{4}, \infty) \subseteq f_n^{-1}(\frac{2}{3}, \infty)$ and $\norm{f_n} \leq 2/3 + 2^{-n}$.  The following lemma captures the feature of well structured
names that we will exploit in Section \ref{sec:proof.main}; namely, it 
will be used to show that if $a \in X$ has a computable well structured name, then the 
evaluation-at-$a$ functional is computable.

\begin{lemma}\label{lm:eval}
Suppose $(f_s)_{s \in \N}$ is a well structured name of $a \in X$, and let $g \in C(X; [0,1])$.
Then, the following are equivalent.
\begin{enumerate}
	\item There exists $s \in \N$ so that 
$\norm{f_s(1 - g)} < \frac{1}{3}$.\label{lm:eval::2}
\item $g(a) > \frac{1}{2}$.\label{lm:eval::1}
	
	\item $\norm{f_s(1 - g)} < \frac{1}{3}$ for all sufficiently large $s \in \N$.\label{lm:eval::3}
\end{enumerate}
\end{lemma}

\begin{proof}
First suppose $\norm{f_s (1 - g)} < \frac{1}{3}$ for some $s\in \N$.
Since $f_s(a) > \frac{2}{3}$, we must have $g(a) > \frac{1}{2}$.  

Next suppose 
$g(a) > \frac{1}{2}$, and set 
$\epsilon = \frac{1}{2}(g(a) - \frac{1}{2})$.  Let $V = g^{-1}(\frac{1}{2} + \epsilon, \infty)$.
We first show that for all sufficiently large $s$, $f_s(t)(1 - g(t)) < \frac{1}{3}$ when $t \in V$.
To this end, choose $N_0 \in \N$ so that $(\frac{1}{2} - \epsilon)(\frac{2}{3} + 2^{-N_0}) < \frac{1}{3}$.
Suppose $s \geq N_0$ and $t \in V$.  Then, by the definition of $V$, $1 - g(t) < \frac{1}{2} - \epsilon$.
Since $(f_s)_{s \in \N}$ is well structured, $f_s(t) \leq \frac{2}{3} + 2^{-s}$.  Hence, 
$(1 - g(t))f_s(t) < \frac{1}{3}$.  

Now, we show that for all sufficiently large $s$, $(1 - g(t))f_s(t) \leq \frac{1}{4}$ for all 
$t \in X \setminus V$.  First, set $K_s = \overline{f_s^{-1}(1/4, \infty)}$.  Since $(f_s)_{s\in \N}$
is well structured, $K_{s+1} \subseteq K_s$.  Since $(f_s)_{s \in \N}$ names $a$, 
$a \in \bigcap_s K_s$.  However, since $(f_s)_{s \in \N}$ is well structured, 
$K_{s+1} \subseteq f_s^{-1}(1/2, \infty)$, and so $\bigcap_s K_s = \{a\}$.  Thus, 
$\bigcap_s K_s \setminus V = \emptyset$.  By Cantor's Theorem, $K_s \subseteq V$ for all 
sufficiently large $s$.  If $t \in X \setminus V$, and if $K_s \subseteq V$, then $f_s(t) \leq 1/4$ and
so $f_s(t) (1 - g(t)) \leq \frac{1}{4}$.
\end{proof}

We say that a vector name $(f_n)_{n \in \N}$ is \emph{adequately structured}
if $f_{n+1}^{-1}(\frac{1}{2}, \infty) \subseteq f_n^{-1}(\frac{2}{3}, \infty)$.
Adequately structured names will serve as an intermediate step towards constructing
well structured names.  The process of building a well structured name from one
that is adequately structured is as follows.  
For all $t \in \R$, let 
\[
\psi(t) = \left\{
\begin{array}{cc}
\frac{1}{2}t & t < \frac{1}{2}\\
\frac{5}{2}(t - \frac{1}{2}) + \frac{1}{4} & \frac{1}{2} \leq t < \frac{2}{3}\\
t & t \geq \frac{2}{3}\\
\end{array}
\right.
\]
We now have the following lemma. \footnote{We thank Konstantyn Slutsky for suggesting these names and for the proof 
 of Lemma \ref{lm:adq.well}.}

\begin{lemma}\label{lm:adq.well}
If $(f_s)_{s \in \N}$ is an adequately structured name of $a \in X$, then 
$(\min\{\psi \circ |f_s|, 2/3 + 2^{-s}\})_{s \in \N}$ is a well structured name of $a$.
\end{lemma}

\begin{proof}
Let $\widehat{f}_s = \min\{\psi \circ |f_s|, 2/3 + 2^{-s}\}$. By definition of $\psi$, we have 
$\widehat{f}_{s+1}^{-1}(1/4,\infty) \subseteq \widehat{f}_s^{-1}(2/3, \infty)$.

We now claim that $(\widehat{f}_s)_{s \in \N}$ is a vector name of $a$.  
By the choice of $\psi$, $\widehat{f}_s^{-1}(1/4, \infty) = f_s^{-1}(1/2, \infty)$.  
Thus, $\{a\} = \bigcap_s \widehat{f}_s^{-1}(1/4, \infty)$, and so 
$\bigcap_s \widehat{f}_s^{-1}(1/2, \infty) \subseteq \{a\}$.
At the same time, 
\[
a \in \widehat{f}_{s+1}^{-1}(1/4,\infty) \subseteq \widehat{f}_s^{-1}(2/3, \infty) \subseteq \widehat{f}_s^{-1}(1/2,\infty).
\]
Thus, $a \in \bigcap_s \widehat{f}_s^{-1}(1/2,\infty)$.
By definition, $\norm{\widehat{f}_s} \leq 2/3 + 2^{-s}$. 
Thus, $(\widehat{f}_s)_{s \in \N}$ is well structured.  
\end{proof}

We note that Lemma \ref{lm:eval} fails if $(f_s)_{s \in \N}$ is merely an 
adequately structured name; in particular
the implication of (\ref{lm:eval::2}) by (\ref{lm:eval::1}) fails.

Finally, we will use the following lemma to demonstrate that a 
sequence is dense in $X$ by relating it to the density of a sequence in 
$C(X; [0,1])$.  The proof is a standard argument
via Urysohn's Lemma.  

\begin{lemma}\label{lm:dense}
Suppose $(g_n)_{n \in \N}$ is dense in $C(X; [0,1])$, and fix $r \in (0,1)$.
Furthermore, suppose 
$(p_n)_{n \in \N}$ is a sequence of points of $X$ so that 
for each $n \in \N$, if $\norm{g_n} > r$, then 
there exists $k \in \N$ so that $g_n(p_k) > r$. 
Then $(p_n)_{n \in \N}$ is dense in $X$.
\end{lemma}

\begin{proof}
Let $t_0 \in X$, and let $\epsilon > 0$.  By Urysohn's Lemma, there is a continuous 
$\lambda : X \rightarrow [0,1]$ so that $\lambda(t) = 1$ when $t \in \overline{B}(t_0; \epsilon/2)$ and 
$\lambda(t) = 0$ when $t \in X \setminus B(t_0; \epsilon)$.  
Hence, there exists $n$ so that $g_n(t) > r$ when $t \in \overline{B}(t_0; \epsilon / 2)$ and 
$g_n(t) < r$ when $t \in X \setminus B(t_0; \epsilon)$. 
Take $k \in \N$ so that $g_n(p_k) > r$.  Then 
$p_k \in B(t_0; \epsilon)$, establishing the desired conclucion.
\end{proof}

\section{Computability-theoretic preliminaries}\label{sec:prelim.comp}

We begin by addressing the computability of the unit.  These considerations will have some
effect on the uniform computability of certain post-composition operators
and in turn will influence the uniformity of Theorem \ref{thm:main}.  

Suppose $A$ is a unital $C^*$ algebra.  We say that $A$ is \emph{computably unital} if 
$\unit_A$ is a computable vector of every computable presentation of $A$.  
We say that $A$ is \emph{uniformly computably unital} if an $A^\#$-index of $\unit_A$ can be 
computed from an index of $A^\#$.  

The following has been proven by A. Fox.  We include a proof for the sake of completeness.

\begin{theorem}\label{thm:comp.unit}
Every commutative unital $C^*$ algebra is computably unital.
\end{theorem}

\begin{proof}
Let $A$ be a commutative unital $C^*$ algebra.  
Suppose $A^\#$ is a computable presentation of $A$.  
Let $\delta_0 = \frac{1}{2}(1 - 2^{-1/2})$, and fix a rational vector $v_0$ of $A^\#$ so that 
$\norm{v_0 - \unit_A} < \delta_0$.  

Let $k \in \N$.  It is required to compute a rational vector $v$ so that $\norm{v - \unit_A} < 2^{-k}$.
Set $\epsilon_0 = 2^{-(2k + 1)}$.   
 Search for a rational vector $v$ so that 
$\norm{v^2 - v} < \epsilon_0$ and so that $\norm{v - v_0} < \delta_0$.  Since 
the rational vectors are dense in $A$ and the norm is continuous, it follows that this search terminates.
It remains show that $\norm{v - \unit_A} < 2^{-k}$.  By way of contradiction, suppose 
$\norm{v - \unit_A} \geq 2^{-k}$.  Without loss of generality, suppose 
$A = C^*(X)$ for some compact metrizable space $X$.  Hence, there exists $t_0 \in X$ so that 
$|v(t_0) -1| \geq 2^{-k}$.  Let $\alpha = v(t_0)$, and let $c = v(t_0)^2 - v(t_0)$.  
Let $\beta \in \C$ be the other root of $z^2 - z - c$.  
Thus, $\alpha\beta = -c$ and $\beta = 1 - \alpha$.
Since $|c| < \epsilon_0$, $\min\{|\alpha|, |\beta|\} < \sqrt{\epsilon_0}$.  
However, as $\sqrt{\epsilon_0} < 2^{-k} \leq |\alpha - 1|$, $|\alpha| < \sqrt{\epsilon_0}$.  Hence, 
\begin{eqnarray*}
|v(t_0) - v_0(t_0)| & \geq & ||v(t_0)| - |v_0(t_0)|| \\
& \geq & |v_0(t_0)| - |v(t_0)|\\
& \geq & 1 - \delta_0 - \sqrt{\epsilon_0}\\
& \geq & \delta_0.
\end{eqnarray*}
This is a contradiction.
\end{proof}

The proof of Theorem \ref{thm:comp.unit} is nonuniform.  
As we shall see, our only obstacle to a fully uniform proof of Theorem \ref{thm:main} is the uniform
computability of the unit.  Hence, we now explore how uniform 
can Theorem \ref{thm:comp.unit} be made.  
Our best result in this direction is the following.  

\begin{proposition}\label{prop:fin.conn}
Let $X$ be a compact metrizable space, and suppose $X$ has a finite number of connected components.
Then, $C^*(X)$ is uniformly computably unital.
\end{proposition}

\begin{proof}
Let $A = C^*(X)$.  Let $n_0$ denote the number of connected components of $X$.  
Let $C_1, \ldots, C_{n_0}$ denote the connected components of $X$.  Fix a computable presentation 
$A^\#$ of $A$.  Let $k \in \N$, and suppose it is required to compute a rational vector $v_0$ 
of $A^\#$ so that $\norm{v_0 - \unit_A} < 2^{-k}$.  
Let:
\begin{eqnarray*}
k_0 & = & k + 3\\
\epsilon_1 & = & \frac{1}{2} \min\{(1 - 2^{-k_0}n_0)^2, 2^{-(k_0 + 1)}/n_0\}\\
\epsilon_2 & = & \frac{1}{2} \min\{\frac{1}{4}, (2^{-k_0}/n_0)^2\}.
\end{eqnarray*}
Search for rational vectors $v_1, \ldots, v_{n_0}$ that satisfy the following conditions:  
\begin{enumerate}
	\item $|\norm{v_n} - 1| < \frac{1}{2}$.\label{c1}

	\item $\norm{v_mv_n} < \epsilon_1$ when $m \neq n$.\label{c2}
	
	\item $\norm{v_n^2 - v_n} < \epsilon_2$.\label{c3}
\end{enumerate}
Thus, $\norm{v_n} > \frac{1}{2}$.  Set $v = \sum_{n = 1}^{n_0} v_n$.

Let $I_n$ denote the indicator function of $C_n$.  By letting $v_n$ approach $I_n$ for each $n$, it is seen that this 
search terminates.  It only remains to show $\norm{v - \unit_A} < 2^{-k}$.  

Let us say that $v_m$ is \emph{approximately supported on $C_n$} if 
$|v_m(t) - 1| < 2^{-k_0}/n_0$ for all $t \in C_n$.  
We first
claim that for each $m \in \{1, \ldots, n\}$, there exists $n \in \{1, \ldots, n_0\}$ so that 
$v_m$ is approximately supported on $C_n$.   Towards this end, take $t_0 \in X$ so that 
$|v_m(t_0)| > \frac{1}{2}$, and take the unique $n \in \{1, \ldots, n_0\}$ so that $t_0 \in C_n$.

For each $t \in C_n$, let $c(t) = v_m(t)^2 - v_m(t)$.  For each $t \in C_m$, let 
$\alpha(t) = v_m(t)$, and let $\beta(t)$ be the other root of $z^2 - z - c(t)$.  Thus, 
$\alpha(t) + \beta(t) = 1$, and $\alpha(t)\beta(t) = -c(t)$.  By (\ref{c3}), $|c(t_0)| < \epsilon_2$.
Since $|v_m(t_0)| > \frac{1}{2}$, $|v_m(t_0)| > \sqrt{\epsilon_2}$. 
Thus, $|1 - v_m(t_0)| = |\beta(t_0)| < \sqrt{\epsilon_2} < 2^{-k}/n_0$.  

By way of contradiction, suppose $t_1 \in C_n$ and $|1 - v_m(t_1)| \geq 2^{-k}/n_0$.  
Thus, $|\beta(t_1)| = |1 - v_m(t_1)| > \sqrt{\epsilon_2}$.  Hence, $|\alpha(t_1)| = |v_m(t_1)| < \sqrt{\epsilon_2}$.  
Since $|v_m(t_0)| > \frac{1}{2} > \sqrt{\epsilon_2}$, by connectedness there exists $t_2 \in C_n$ so that 
$1/2 > |v_m(t_2)| > \sqrt{\epsilon_2}$.  Thus, $|1 - v_m(t_2)| > \frac{1}{2}$.  
Putting all this together, we obtain $|v_m(t_2)(v_m(t_2) - 1)| > \sqrt{\epsilon_2}\cdot\frac{1}{2} > \epsilon_2$, yielding
a contradiction. 

Now, we claim that for each $m \in \{1, \ldots, n_0\}$, there is exactly one $n \in \{1, \ldots, n_0\}$ 
so that $v_m$ is approximately supported on $C_n$.
By way of contradiction suppose otherwise.  By our first claim and the pigeonhole principle, 
there exist $m,m',n \in \{1, \ldots, n_0\}$ so that 
$m \neq m'$ and $v_m$, $v_{m'}$ are approximately supported on $C_n$.  
Then, $\min\{|v_m(t_0)|, |v_{m'}(t_0)|\} > 1 - 2^{-k}/n_0$.
So, $|v_m(t_0) v_{m'}(t_0)| > \epsilon_1$.  But, $\norm{v_m v_{m'}} < \epsilon_1$- a contradiction.

We now conclude there is an injective map $j : \{1, \ldots, n_0\} \rightarrow \{1, \ldots, n_0\}$
so that $v_m$ is approximately supported on $C_{j(m)}$ for each $m \in \{1, \ldots, n_0\}$.
Hence, $j$ is surjective and $\unit_A = \sum_{m = 1}^{n_0} I_{j(m)}$.  
Therefore, 
\begin{eqnarray*}
\norm{v - \unit_A} & \leq & \sum_{m = 1}^{n_0} \norm{v_m - I_{j(m)}}
\end{eqnarray*}
We estimate $\norm{v_m - I_{j(m)}}$ as follows. Let $m \in \{1, \ldots, n_0\}$.  If $t \in C_{j(m)}$, then 
$|v_m(t) - I_{j(m)}(t)| < 2^{-k_0}/n_0$.
Suppose $t \not \in C_{j(m)}$.  Let $t \in C_n$, and let $m' = j^{-1}(n)$.  Thus $m' \neq m$. 
By definition of $j$, $|v_{m'}(t) - 1| < 2^{-k_0}/n_0$.  
Thus, by (\ref{c1})
\begin{eqnarray*}
|v_m(t) v_{m'}(t) - v_m(t)| & \leq & (1 + \frac{1}{2})(2^{-k_0}/n_0)\\
& < & 2^{-k_0 + 1}/n_0.
\end{eqnarray*}
Hence, 
\begin{eqnarray*}
|v_m(t)| & \leq & |v_m(t) - v_m(t)v_{m'}(t)| + |v_m(t) v_{m'}(t)| \\
& \leq & 2^{-k_0 + 1}/n_0 + \epsilon_1.
\end{eqnarray*}
Thus, $\norm{v_m - I_{j(m)}} < 2^{-k_0 + 1}/ n_0 + \epsilon_1$.  
Hence, 
\[
\norm{v - \unit_A} \leq 2^{-k_0 + 1} + n_0 \epsilon_1 < 2^{-k_0 + 2} < 2^{-k}.
\]
\end{proof}

We now turn to the computability of post-composition operators.  We start with post-composition
operators induced by rational $*$-polynomials.  We summarize our results in the following proposition.
We believe these findings are simple enough so as not to require a formal proof.  At the same time, 
we believe they are useful enough to warrant a formal statement. 
 
 \begin{proposition}\label{prop:post.poly}
 Suppose $A = C^*(X)$, and let $A^\#$ be a computable presentation of $A$.  
 \begin{enumerate}
 	\item If $p : \C^n \rightarrow \C$ is a rational $*$-polynomial, then the 
  post-composition operator $(f_1, \ldots, f_n) \mapsto p(f_1, \ldots, f_n)$
  is a computable operator of $A^\#$.  

        \item In addition, if $p(\vec{0}) = 0$, then an index of this operator can be computed from 
        $p$ and an index of $A^\#$. 

        \item If $p(\vec{0}) \neq 0$, then an index of this operator can be computed from 
        $p$, an index of $A^\#$, and an $A^\#$-index of $\unit_A$.  
 \end{enumerate}
 \end{proposition}
 
We now use Proposition \ref{prop:post.poly} to establish the conditions under which a post-composition
 operator is computable.  

\begin{proposition}\label{prop:post.comp}
 Suppose $A = C^*(X)$, and let $A^\#$ be a computable presentation of $A$.  
 \begin{enumerate}
 	\item If $g : \C^n \rightarrow \C$ is computable, then the 
  post-composition operator $(f_1, \ldots, f_n) \mapsto g(f_1, \ldots, f_n)$
  is a computable operator of $A^\#$.  

        \item In addition, if $g(\vec{0}) = 0$, then an index of this operator can be computed from 
        an index of $g$ and an index of $A^\#$. 

        \item If $g(\vec{0}) \neq 0$, then an index of this operator can be computed from 
        an index of $g$, an index of $A^\#$, and an $A^\#$-index of $\unit_A$.  
 \end{enumerate}
\end{proposition}

\begin{proof}
Given $A^\#$-names of $f_1, \ldots, f_n \in A$, we compute an $A^\#$-name of 
$g(f_1, \ldots, f_n)$ as follows.
First, compute a positive integer $M$ so that $\max\{\norm{f_1}, \ldots, \norm{f_n}\} < M$.  
Let $R_M = \{(z_1, \ldots, z_n) \in \C^n\ :\ |z_j| \leq M\}$.  

We then compute a sequence $(p_k)_{k \in \N}$ of rational $*$-polynomials so that 
$|p_k(q) - g(q)| < 2^{-k}$ whenever $q \in R_M$.  Thus, for each $k \in \N$,  
$\norm{p_k(f_1, \ldots, f_n) - g(f_1, \ldots, f_n)} < 2^{-k}$.   
By Proposition \ref{prop:post.poly}, 
it follows that we may compute for each $k$ a rational vector 
$u_k$ so that $\norm{u_k -p_{k+1}(f_1, \ldots, f_n)} < 2^{-(k+1)}$.  Thus, $(u_k)_{k \in \N}$ is an 
$A^\#$-name of $g(f_1, \ldots, f_n)$. 

By inspection, all of the steps in the above procedure are uniform in $n$, an index of $g$, 
an index of $A^\#$, and an $A^\#$-index of $\unit_A$.  
If $g(\vec{0}) = 0$, then we can choose $p_k$ so that $p_k(\vec{0}) = 0$.  
Hence, in this case, only an index of $g$ and an index of $A^\#$ are required to compute
an index of the post-composition operator induced by $g$. 
\end{proof}

\begin{corollary}\label{cor:inv.abs.mx.mn}
If $X$ is a compact metrizable space, then $|\ |$, $\max$, and $\min$ are uniformly intrinsically 
computable operators of $C^*(X)$.
\end{corollary}

\section{Proof of Theorem \ref{thm:main} }\label{sec:proof.main}

Suppose $X$ is a compact metrizable space, and let $A =C^*(X)$.  
For each $p \in X$ and $f \in A$, let $\widehat{p}(f) = f(p)$, so $\widehat{p}$ is the evaluation-at-$p$ functional.  
 
Fix a computable presentation 
$A^\#$ of $A$.  In order to simplify exposition, throughout this section, all computability is referent to this particular presentation.  Consequently, 
when we say that a vector, sequence, or operator is computable, we mean that it is a computable 
vector, sequence, or operator of $A^\#$.  

By standard techniques, we may compute an effective and injective enumeration
$(\kappa_n)_{n \in \N}$ of a dense sequence of rational vectors of $A^\#$.  
For all $a,b \in X$, let 
\[
d(a,b) = \sum_{n \in \N} 2^{-n} \frac{|\kappa_n(a) - \kappa_n(b)|}{1 + |\kappa_n(a) - \kappa_n(b)|}.
\]
It is well-known that $d$ is a metric that is compatible with the topology of $X$. 
For all $p,q \in X$, let $D_p(q) = d(p,q)$.  Thus, $D_p \in A$ 
for each $p \in X$. 

Our goal now is to build a computably compact presentation of $(X,d)$.  
The following theorem reduces the complexity of this task while
also motivating the method of our proof.

\begin{theorem}\label{thm:comp.fnctl}
Suppose $(a_n)_{n \in \N}$ is a dense sequence of points of $X$ so that 
$\widehat{a_n}$ is a computable functional uniformly in $n$.  
Then, $(D_{a_n})_{n \in \N}$ is a computable sequence of vectors and 
$(X, d, (a_n)_{n \in \N})$ is a computably compact presentation of $X$.
\end{theorem}

\begin{proof}
We first note that 
\begin{equation}
D_{a_n}  =  \sum_{t \in \N} 2^{-t} \frac{|\kappa_t - \kappa_t(a_n) \cdot \unit_A|}{\unit_A + |\kappa_t - \kappa_t(a_n) \cdot \unit_A|}.\label{eqn:gn.vec}
\end{equation}
It is easy to verify that the series in Equation (\ref{eqn:gn.vec}) (viewed as a sequence of partial sums)
 has a computable modulus of 
convergence.  
Thus, we only need to demonstrate that each term of this series is computable
uniformly in $n,t$.  To facilitate this, for all $n,t \in \N$, set:
\begin{eqnarray*}
u_{n,t} & = & |\kappa_t - \kappa_t(a_n) \cdot \unit_A| \\
v_{n,t} & = & \unit_A + u_{n,t}.\\
\end{eqnarray*}
We note that since $v_{n,t}(p) \geq 1$ for all $p \in X$, $v_{n,t}$ is invertible and 
$\norm{v_{n,t}^{-1}} \leq 1$.
By Theorem \ref{thm:comp.unit} and Corollary \ref{cor:inv.abs.mx.mn}, 
$u_{n,t}$ and $v_{n,t}$ are computable vectors uniformly in $n,t$.  It only remains to 
show that $v_{n,t}^{-1}$ is computable uniformly in $n,t$.  This can be accomplished 
by a simple search procedure as follows.  Given $k \in \N$, search for a rational vector $u$
so that $\norm{u v_{n,t} - \unit_A} < 2^{-k}$.  Since $\norm{v_{n,t}^{-1}} \leq 1$, it
follows that $\norm{u - v_{n,t}^{-1}} < 2^{-k}$.

Since $\widehat{a_n}$ is computable uniformly in $n$, it now follows that 
$(d(a_m,a_n))_{m,n \in \N}$ is a computable array of real numbers.  
whence it follows that $X^\# = (X,d, (a_n)_{n \in \N})$ is a computable presentation of $X$.
All that remains is to demonstrate the computable compactness of $X^\#$.  
We accomplish this as follows.  
For every $k \in \N$ and finite nonempty $F \subseteq \N$, 
$X = \bigcup_{n \in F} B(a_n; 2^{-k})$ if and only if 
$\norm{\min_{n \in F} D_{a_n}} < 2^{-k}$.  Since $(a_n)_{n \in \N}$ is dense in $X$, 
and since $X$ is compact, for each $k \in \N$, there \emph{is} a finite $F \subseteq \N$
so that $X = \bigcup_{n \in F} B(a_n; 2^{-k})$.  By Corollary \ref{cor:inv.abs.mx.mn}, 
from $F$ it is possible to compute $\norm{\min_{n \in F} D_{a_n}}$.  
It follows that $X^\#$ is a computably compact presentation.
\end{proof}

Thus, we now seek to produce a dense sequence $(a_n)_{n \in \N}$ of points of $X$ 
so that $\widehat{a_n}$ is computable uniformly in $n$. To this end, for each $n \in \N$, set $u_n = \min\{|\kappa_n|, \unit_A\}$.  
It follows from Corollary \ref{cor:inv.abs.mx.mn} that $(u_n)_{n \in \N}$ is computable.
Furthermore, when $f \in C(X; [0,1])$, $\norm{u_n - f} \leq \norm{|\kappa_n| - f} \leq \norm{\kappa_n - f}$.  
Thus, $(u_n)_{n \in \N}$ is dense in $C(X; [0,1])$.  It then follows that 
$(u_n)_{n \in \N}$ generates a dense subalgebra of $C^*(X)$.

Ensuring the computability of $\widehat{a_n}$ is simplified somewhat by the following 
principle which is perhaps of independent interest.

\begin{proposition}\label{prop:comp.eval}
Suppose $p \in X$ has a computable well structured name.  
Then, $\widehat{p}$ is computable.  Furthermore, an index of $\widehat{p}$ can 
be computed from an index of a well structured name of $p$.
\end{proposition}

\begin{proof}
Fix a well structured name $(f_s)_{s \in \N}$ of $p$.  
Since $\norm{\widehat{p}} \leq 1$, and since $(u_m)_{m \in \N}$ generates a dense subalgebra 
of $A$, it suffices to show $(\widehat{p}(u_m))_{m \in \N}$ is computable.
  
Let $(q_r)_{r \in \Q\cap(0,1)}$ be a computable family of polynomials with the following properties:
\begin{itemize}
	\item $q_r$ maps $[0,1]$ into $[0,1]$.
	
	\item $q_r(x) > \frac{1}{2}$ if and only if $x > r$.
\end{itemize}

Let $k,m \in \N$.  It is required to compute a rational number $q$ so that 
$|\widehat{p}(u_m) - q| < 2^{-k}$.   We search for $r_0, r_1 \in \Q \cap [0,1]$ and $s \in \N$ so that the following hold.
\begin{enumerate}
	\item $0 < r_1 - r_0 < 2^{-k}$.
	
	\item $r_0 = 0$ or 
$\norm{f_s(1 - q_{r_0} \circ u_m)} < \frac{1}{3}$.

	\item $r_1 = 1$ or $\norm{f_s(1 - q_{1 - r_1} \circ (1 - u_m))} < \frac{1}{3}$.
\end{enumerate}
Set $q = r_0$.  

We first show that this search terminates.  To this end, we first consider the case $u_m(p) = 0$.  
If we take $r_0 = 0$ and $r_1 = 2^{-(k+1)}$, then $q_{1 - r_1}(1 - u_m(p)) > \frac{1}{2}$ and so 
by Lemma \ref{lm:eval} there exists $s$ so that $\norm{f_s(1 - q_{1 - r_1} \circ (1 - u_m))} < \frac{1}{3}$.
Now, suppose $u_m(p) = 1$.   In this case, we take $r_1 = 1$ and $r_0 = 1 - 2^{-(k+1)}$.  
We then have $q_{r_0}(u_m(p)) > 1/2$ and so there exists $s$ so that 
$\norm{f_s(1 - q_{r_0} \circ (u_m))} < \frac{1}{3}$.  Finally, suppose $0 < u_m(p) < 1$.  
Choose rational numbers $r_0, r_1 \in (0,1)$ so that $0 < r_1 - r_0 < 2^{-k}$ and so that 
$r_0 < u_m(p) < r_1$.  Thus, 
$q_{r_0}(u_m(p)) > \frac{1}{2}$ and $q_{1 - r_1}(1-u_m(p)) > \frac{1}{2}$.  
It then follows from Lemma \ref{lm:eval} that there exists $s$ so that 
$\norm{f_s(1 - q_{r_0} \circ u_m)} < \frac{1}{3}$ and so that 
$\norm{f_s(1 - q_{1 - r_1} \circ (1 - u_m))} < \frac{1}{3}$.  Hence, in all cases, the above search
terminates.

It only remains to show $|q - u_m(p)| < 2^{-k}$.  It follows from Lemma \ref{lm:eval} that 
$r_0  \leq u_m(p) \leq r_1$.  Since $r_1 - r_0 < 2^{-k}$, we have $|q - u_m(p)| < 2^{-k}$.
\end{proof}

In light of Lemma \ref{lm:dense} and Proposition \ref{prop:comp.eval}, much of the 
task at hand now reduces to the following.  

\begin{lemma}\label{lm:adm.gn}
If $f \in C(X; [0,1])$ is a computable vector, and if $\norm{f} > \frac{2}{3}$, then there is a computable adequately structured name that begins with $f$.   Furthermore, an index of such a name can be computed from 
an index of $f$.
\end{lemma}

\begin{proof}
Fix a polynomial $p$ over $\Q$, e.g., $p(x) = 16 x^5  - 40 x^4 + 32 x^3 - 8x^2 + x$, with the following properties.
\begin{enumerate}
	\item $p$ maps $[0,1]$ onto $[0,1]$.
	
	\item For all $x \in \R$, $p(x) = x$ if and only if $x \in \{0,\frac{1}{2}, 1\}$.
	
	\item $p(x) > x$ when $x \in (-\infty,0) \cup (\frac{1}{2},1)$.
	
	\item $p(x) < x$ when $x \in (0, \frac{1}{2}) \cup (1, \infty)$.  
\end{enumerate}
Let $p^k$ denote the $k$-th iterate of $p$.   
It follows that for each $x \in [0,1]$,
\[
\lim_k p^k(x) = \left\{
\begin{array}{cc}
1 & \frac{1}{2} < x \leq 1\\
\frac{1}{2} & x = \frac{1}{2}\\
0 & 0 \leq x < \frac{1}{2}.\\
\end{array}
\right.
\]
We construct $(f_s)_{s \in \N}$ by stages.  At stage $s$, we define $f_s$, and we may declare a $j \in \N$
to be \emph{incorporated}.\\

\noindent\bf Stage $0$:\rm\ Set $f_0 = f$.  No $j$ is incorporated at stage $0$.\\

\noindent\bf Stage $s+1$:\rm\ We first define a function $h_s$ by cases as follows.  
If there is no $j \leq s$ so that $\norm{f_s + u_j} > \frac{5}{3}$ and so that 
$j$ has not been incorporated by the end of stage $s$, then 
set $h_s = \frac{3}{4}f_s$.  Otherwise, let $j_s$ be the least such $j$, and set 
$h_s = \frac{3}{10}(f_s + u_{j_s})$; we also incorporate $j_s$ at stage $s+1$.  

By way of induction, $h_s \in C(X; [0,1])$ and $2/3 < \norm{f_s} \leq 1$.  Thus, $1 > \norm{h_s} > \frac{1}{2}$.  
By the properties of $p$, there is a natural number $k$ so that $\norm{p^k \circ h_s} > \frac{3}{4}$; 
let $k_s$ be the least such number.  
Set $f_{s+1} = p^{k_s} \circ h_s$.

We now show $f_{s+1}^{-1}(1/2,\infty) \subseteq f_s^{-1}(2/3, \infty)$.  Suppose $f_{s+1}(t) > \frac{1}{2}$.  
By the properties of $p$, $h_s(t) > \frac{1}{2}$.
If no $j$ is incorporated at $s+1$, then $h_s = \frac{3}{4} f_s$ and so $f_s(t) > \frac{2}{3}$.
Suppose $j_s$ is incorporated at $s+1$.  Then, $f_s(t) + u_{j_s}(t) > \frac{5}{3}$, and so 
$f_s(t) > \frac{2}{3}$.  

We now demonstrate that $(f_s)_{s \in \N}$ names a point.  By construction, for each $s \in \N$, there exists $x_s \in X$ so that $f_{s+1}(x_s) > \frac{3}{4}$.   
Since $X$ is compact and metrizable, there is an $a \in X$ and an increasing sequence $(s_j)_{j \in \N}$ so that
$\lim_j x_{s_j} = a$.  
We show that $(f_s)_{s\in \N}$ names $a$.  
To this end, let $C = \bigcap_{s \in \N} f_s^{-1}(1/2, \infty)$.   We note that $C = \bigcap_{s \in \N} \overline{f_s^{-1}(1/2,\infty)}$.
By a standard argument, $a \in C$.  By way of contradiction suppose $b \in C - \{a\}$.   
By Urysohn's Lemma, there is a continuous $\lambda : X \rightarrow [0,1]$ so that 
$\lambda(a) = 1$ and $\lambda (b) = 0$.  Hence, there exists $r$ so that 
$u_r(a) > \frac{11}{12}$ and so that $u_r(b) < \frac{2}{3}$.   Let $t_0$ be the least stage so that 
every $r' < r$ that is incorporated has been incorporated by the end of stage $t_0$.  
Let $m$ be the least integer so that $s_{m} + 1> t_0$.  Then, 
$u_r(a) + f_{s_{m} + 1}(a) > \frac{3}{4} + \frac{11}{12} = \frac{5}{3}$.  Thus, 
if $r$ has not been incorporated by stage $s_{m} + 1$, it is at stage $s_{m} + 2$.   
So, let $s+1$ be the stage at which $r$ is incorporated.  Then, 
$h_s(b) <  \frac{1}{2}$, and so $f_{s+1}(b) < \frac{1}{2}$- a contradiction since $b \in C$.
\end{proof}

There is a computable $e : \N \rightarrow \N$ so that $\ran(e) = \{n \in \N\ :\ \norm{u_n} > \frac{2}{3}\}$.
It follows from Lemma \ref{lm:adm.gn} that there is a uniformly computable sequence 
$(\Lambda_n')_{n \in \N}$ of adequately structured names so that $\Lambda_n'$ 
originates with $u_{e(n)}$.
Let $a_n \in X$ be the point named by $\Lambda_n'$.
By Lemma \ref{lm:dense}, $(a_n)_{n \in \N}$ is dense in $X$ (take $r = \frac{2}{3})$.
Now, set 
$\Lambda_n = \min\{\psi \circ \Lambda_n', 2^{-n}\}$ where $\psi$ is the function defined in 
Section \ref{sec:prelim.cl}.  Thus, by Lemma \ref{lm:adq.well}, $\Lambda_n$ also names 
$a_n$.  In addition, by Proposition \ref{prop:post.comp}, 
$(\Lambda_n)_{n \in \N}$ is computable.  
Hence, by Proposition \ref{prop:comp.eval}, $(\widehat{a_n})_{n \in \N}$ is computable.
Therefore, by Theorem \ref{thm:comp.fnctl}, $X$ has a computably compact presentation.

\section{Applications and extensions}\label{sec:appl.ext}

Since isometric isomorphisms map computable presentations onto computable presentations,
 Theorem \ref{thm:main} easily implies the following.

\begin{corollary}\label{cor:spec}
Suppose $A$ is a unital commutative $C^*$ algebra.  If $A$ is computably presentable, 
then the spectrum of $A$ has a computably compact presentation.
\end{corollary}

In addition, we can now give nontrivial examples of operator algebras that do not have 
computable presentations.

\begin{corollary}\label{cor:appl} \mbox{ }
\begin{enumerate}
	\item[(C1)] There exists a compact space $X$ so that $X$ is computably presentable but $C^*(X)$ has no computable presentation.

\item[(C2)] There exists a commutative unital $C^*$-algebra that has a $Y$-computable presentation if and only if $Y$ is not low (i.e., $Y' > \emptyset'$).
\end{enumerate}
\end{corollary}
	
(C1) follows from the existence of a compact computably presentable
metric space $X$ that is not homeomorphic to any metric space with a computably compact 
presentation \cite{lupini, EffedSurvey}.
	Notably, perhaps the most elegant way to produce such a space uses (D2) stated in 
 Section \ref{sec:back}, and one 
	more effective duality established  in \cite{newpolish} which we will not state here.
	Using this sequence of effective dualities, the existence of such a space $X$ can be reduced to the 
	old result of Mal'cev~\cite{Mal} characterising  computable subgroups of $(\mathbb{Q}, +)$.
	Similarly, in (C2) the existence of such a pathological space can be reduced to the existence of a torsion-free abelian group that has exactly the non-low presentations established in \cite{mel}; see also \cite{newpolish}
	for a detailed explanation.

By inspection, the steps of the proof of Theorem \ref{thm:main} are uniform.  That is, from an index of a presentation $C^*(X)^\#$ and a $C^*(X)^\#$-index of the unit, we may compute 
an index of a computably compact presentation of $X$.  
However, uniformity holds in a more general sense which we describe now.

To begin, we set forth a method of naming presentations of $C^*$ algebras.
Let $A$ be a $C^*$ algebra, and fix a presentation $A^\# = (A, (u_n)_{n \in \N})$.
The set $D(A^\#) = \{(r, j, r')\ :\ r,r' \in \Q \cap (0, \infty)\ \wedge\ r < \norm{\ratpoly_j[A^\#]} < r'\}$ is the \emph{diagram} of 
$A^\#$.  We say that $f \in \N^\N$ is a \emph{name} of $A^\#$ if 
$D(A^\#) = \{(r,j,r')\ :\ \exists k \in \N f(k) = \langle r,j,r' \rangle\}$.  
It follows that $A^\#$ is computable if and only if it has a computable name.  

Suppose $\mathcal{M} = (X,d)$ is a metric space and $\mathcal{M}^\# = (\mathcal{M}, (s_j)_{j \in \N})$
is a presentation of $\mathcal{M}$. 
The \emph{diagram} of $\mathcal{M}^\#$ is $D(\mathcal{M}^\#) = \{(r,j,k,r')\ :\ r < d(s_j,s_k) < r'\}$.
 A \emph{name} of $\mathcal{M}^\#$ is an $f \in \N^\N$ so that 
$D(\mathcal{M}^\#) = \{(r,j,k,r')\ :\ \exists t \in \N f(t) = \langle r,j,k,r' \rangle\}$.

A \emph{total boundedness function} of $\mathcal{M}^\#$ is an $f \in \N^\N$ so that for each $j \in \N$, $f(j)$ is a code of a 
finite $F \subseteq \N$ so that $X = \bigcup_{k \in F} B(s_k; 2^{-j})$.  
A computably presented metric space is compact if and only if it has a total boundedness function.

Now, we note that the steps in the proof of Theorem \ref{thm:main} are fully uniform in that 
they require only the name of a presentation $C^*(X)^\#$ and an $C^*(X)^\#$-name of the unit.
  We are thus led to the following.

\begin{theorem}\label{thm:main.uniform}
There is an oracle Turing machine that, given a name of a presentation $A^\#$ of a commutative unital 
$C^*$ algebra $A$ and an $A^\#$-name of $\unit_A$, produces a 
name of a presentation of $\Delta(A)^\#$ and a total boundedness function of 
$\Delta(A)^\#$.
\end{theorem}

We note that A. Fox \cite{Fox.2022+} gives a computable operator which, given the name of a presentation of a compact $X$ and a total boundedness function, produces a name of a presentation of $C(X)$ and a name of the unit. Thus a total boundedness function and a unit are, in some sense, the right data to make the Gelfand duality effective.

In addition, by Theorem \ref{thm:comp.unit}, we obtain the following.

\begin{corollary}\label{cor:conn.unif}
There is a Turing machine that, given the name of a presentation $A^\#$ of a commutative unital 
$C^*$ algebra $A$ whose spectrum has finitely many connected components, and the number of connected 
components of $\Delta(A)$, produces a 
name of a presentation of $\Delta(A)^\#$ and a total boundedness function of 
$\Delta(A)^\#$.
\end{corollary}

\section{Conclusion}\label{sec:concl}

We have shown that the Gelfand duality for $C^*$ algebras holds effectively, and that this 
effective duality is highly uniform provided enough information about the unit is given.
In doing so, we have contributed to the program of studying the effective content of 
classical mathematics by providing another effective duality principle and by advancing the 
computable theory of operator algebras.  The essential content of our main result is that 
from a (presentation of) a $C^*$ algebra $C^*(X)$, one can compute a complete description of the 
points of $X$.  Our theorem connects the computability theory of $C^*$ algebras with the 
well-developed computability theory of Polish spaces \cite{EffedSurvey}.  
Along the way, we have produced a number of less important though useful 
results on the computability of the unit.

Just as the dualities of classical mathematics build bridges between seemingly unrelated
disciplines, effective dualities transfer methods and theorems between disparate 
branches of computable mathematics.  One would then expect Theorem \ref{thm:main}
to lead to new discoveries in the computability theory of operator algebras such as 
Corollary \ref{cor:appl}.

At this point, we are naturally led to the following question which was raised by 
T. McNicholl a few years ago (see also \cite{bastone}).

\begin{question}[McNicholl]
Is Banach-Stone Duality effective?  That is, if the Banach space $C(X)$ has a computable 
presentation, does it follow that $X$ has a 
computably compact presentation?
\end{question}
	
	The effective duality (D1) stated above gives a positive answer in the case when the domain is totally disconnected. It has recently been announced by Melnikov and Ng that
	indeed Banach-Stone Duality \emph{fails} computably, in the sense that there is a compact metrizable
	space $X$ so that the Banach space $C(X)$ has a computable presentation but $X$ does not 
	have a computably compact presentation.	

\bibliographystyle{amsplain}
\bibliography{ourbib}

\end{document}